\documentclass[12pt]{article}

\usepackage{graphicx}
\usepackage{epstopdf}
\usepackage{latexsym,amssymb,amsmath,amsthm,amsfonts}

\newtheorem{thm}{Theorem}
\newtheorem{rmk}{Remark}
\newtheorem{lem}{Lemma}
\newtheorem{cor}{Corollary}

\begin{document}

\noindent \textbf{\large Uniform convergence of compactly supported wavelet expansions\\ of Gaussian random processes\footnotetext{Submitted: \today}}\\

\noindent {\it Short title:} \textbf{ Uniform convergence of wavelet expansions}

\vskip 1cm
\noindent \textbf{\large Yuriy Kozachenko$^a$, Andriy Olenko$^b$\footnote{Address correspondence to Andriy Olenko, Department of Mathematics and Statistics, La Trobe University, Victoria 3086, Australia; E-mail:  a.olenko@latrobe.edu.au} and Olga Polosmak$^a$}

\vskip 5mm
\noindent {$^a$ Department of Probability Theory, Statistics and Actuarial Mathematics, Kyiv University, Kyiv, Ukraine}

\noindent {$^b$ Department of Mathematics and Statistics, La Trobe University, Melbourne, Australia}

\vskip 1cm

\begin{center}
{\large \it Dedicated to Professor Narayanaswamy Balakrishnan on the occasion of\\ the celebration of his 30 years of contributions to statistics.}
\end{center}

\vskip 1cm

\vskip 3mm

\vskip 6mm

\noindent {\bf ABSTRACT}

\noindent \textit{New results on uniform convergence in probability for expansions of Gaussian random processes using compactly supported wavelets  are given. The  main result is valid for general classes of nonstationary processes. An application of the obtained results to stationary processes is also presented. It is shown that the convergence rate of the expansions is exponential.}
\vskip 4mm

\section{Introduction}
In the recent literature a considerable attention was given to wavelet orthonormal series representations of
stochastic processes. Numerous results, applications, and references on convergence of wavelet expansions of random processes in various spaces  can be found in \cite{cam, did, ist, kozol1, kozol2, kur, zha}. While most known results concern the mean-square convergence, for various practical applications one needs to require uniform convergence.

Figures~1 and 2 illustrate wavelet expansions of stochastic processes. Figure~1 presents simulated realizations of a Gaussian process. Contrary to the deterministic case we obtain a new realization and new corresponding reconstruction errors for each simulation. Figure~2 shows a simulated realization of the process and its two wavelet reconstructions with different numbers of terms (see more details in Section~\ref{sec7}). Figure~2 suggests that reconstruction errors become smaller when the number of terms in the expansions increases. Although this effect is well-known for deterministic functions, it has to be established theoretically for different stochastic processes and probability metrics. 
\noindent\begin{figure}[h]
\begin{minipage}{7.5cm}
 \includegraphics[width= 7.5cm,height=6.1cm,trim=0cm 0.5cm 0.5cm 2cm, clip=true]{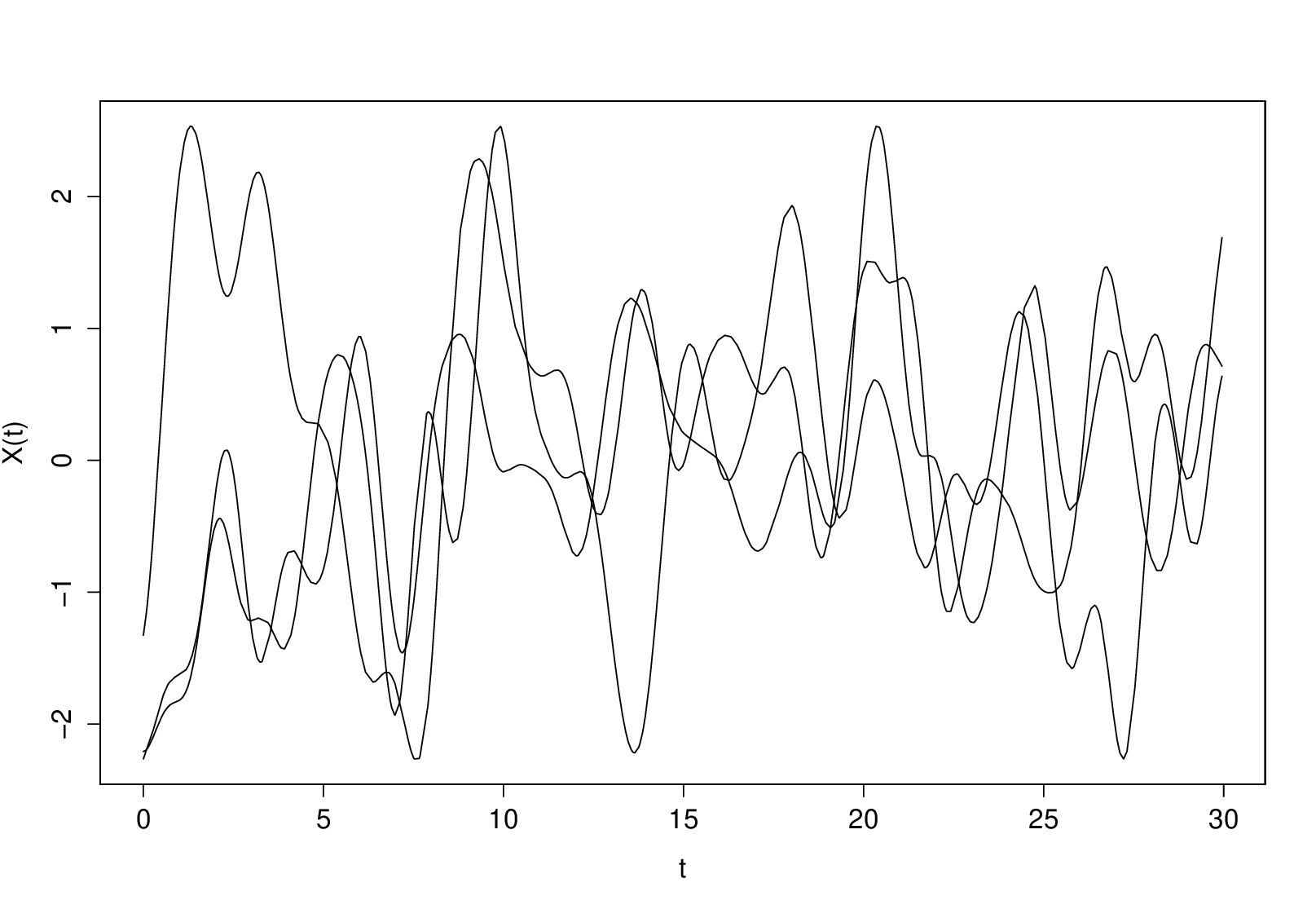}\\
 {{}\vspace{5mm}\textbf{Fig.\,1}\  \small Three realizations of Gaussian process}  \end{minipage}\quad\quad
\begin{minipage}{7.5cm}
 \includegraphics[width= 7.5cm,trim=0cm 0.5cm 0.5cm 2cm, clip=true]{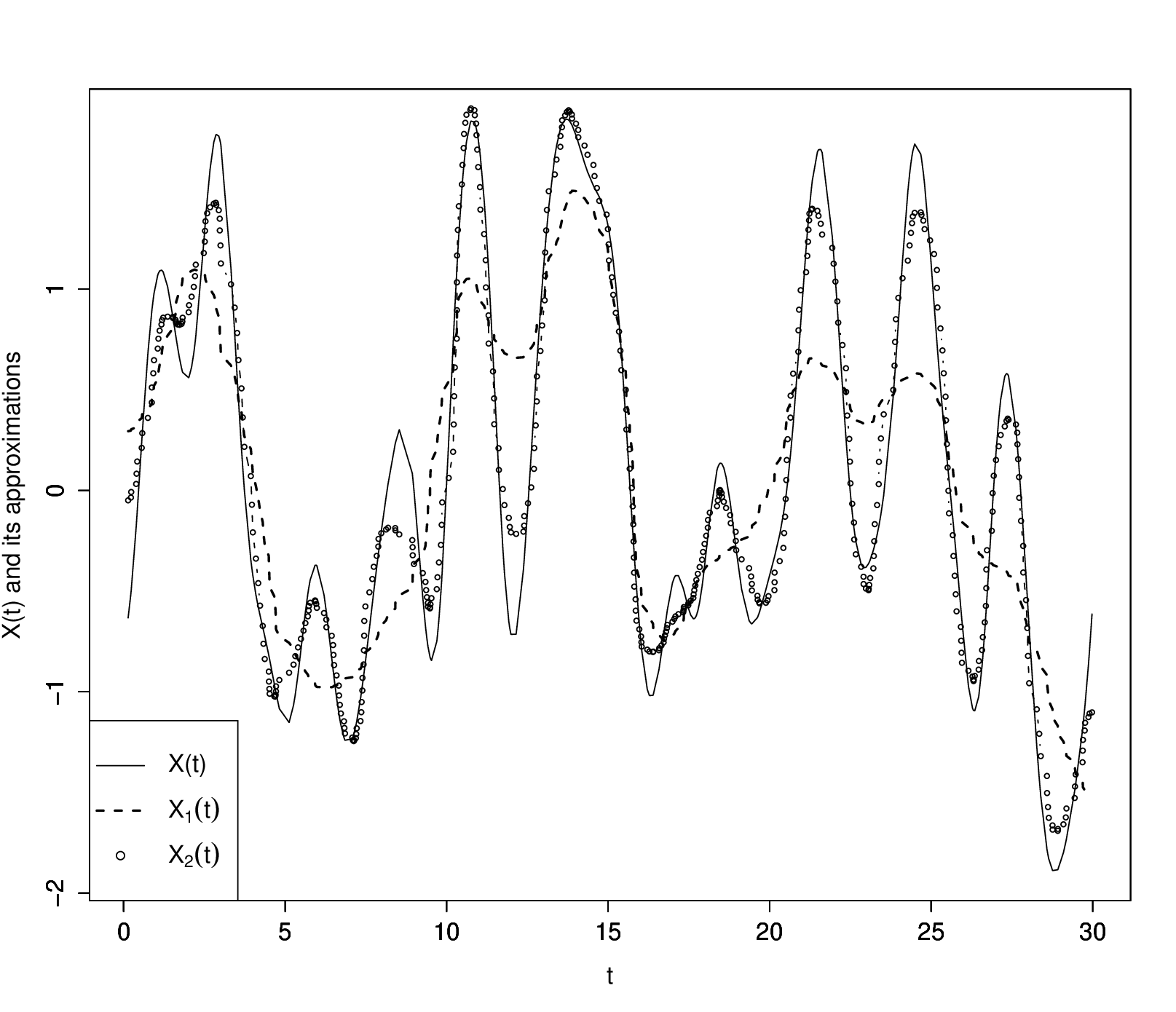}\\
{\textbf{Fig.\,2}\  \small  Gaussian process and its two wavelet reconstructions} \end{minipage}
\end{figure}

In the paper we derive new results on stochastic uniform convergence of wavelet expansions for wide classes of stochastic processes.  Throughout the paper, we assume that the father wavelets have compact supports. It imposes minimal additional conditions on wavelet bases, which can be easily verified. The assumptions  are weaker than those in the former literature,  compare \cite{kozol1, kozol2, kur, wong}. It is easy to verify that numerous wavelets satisfy these conditions, for example, the well known Daubechies, symmlet, and coiflet wavelet bases, see \cite{har}.  We also prove the exponential rapidity of convergence of the wavelet expansions. The  main result is valid for general classes of nonstationary processes. An application of the results to a particular case of stationary processes is also demonstrated.

The obtained results may have various practical applications for the approximation and simulation of random processes, see \cite{bal1,bal2}. Compactly supported wavelet expansions are preferable for practical applications, because, contrary to theoretical results with arbitrary wavelet bases, these expansions have the finite number of terms.  The analysis of the rate of convergence provides a constructive algorithm for determining the number of terms in the wavelet expansions to ensure the approximation of stochastic processes with given accuracy.

The analysis and the approach presented in the paper are new and contribute to the investigations of wavelet expansions of random processes in the former literature.

The organization of the article is the following. In the second section we introduce the necessary background from wavelet theory of non-random functions and discuss various properties of compactly supported wavelets.  In Section 3 we prove mean-square convergence of compactly supported wavelet expansions for the case of stochastic processes. In \S 4 we formulate the main theorem on uniform convergence in probability of compactly supported wavelet expansions of Gaussian random processes. The next section contains the proof of the main theorem. In \S 6 we  obtain the rate of convergence of the wavelet expansions. Finally, an application of the developed technique is shown in section 7.

\section{Wavelet representation of deterministic functions}
In this section we give some basic properties of wavelet expansions of non-random functions. These facts will be used to derive results for stochastic processes in the following sections.

Let $\phi(x),$ $x\in\mathbb R$ be a function from the space
$L_2(\mathbb R)$ such that $\widehat{\phi}(0)\ne 0$ and
$\widehat{\phi}(y)$ is continuous at $0,$ where
$\widehat{\phi}(y)=\int_{\mathbb
R}e^{-iyx}{\phi(x)}\,dx$ is the Fourier transform of $\phi.$

 Suppose that the following assumption
holds true: $\sum_{k\in\mathbb Z} |\widehat{\phi}(y+2{\pi}k)|^2=1\  {\rm (a.e.)}
$

There exists a function $m_0(x)\in L_2([0,2\pi])$, such that
$m_0(x)$ has the period $2\pi$ and
$$\widehat{\phi}(y)=m_0\left(y/2\right)\widehat{\phi}\left(y/2\right)\ {\rm (a.e.)}
$$
In this case the  function $\phi(x)$ is called
the $f$-wavelet.

Let $\psi(x)$ be the inverse Fourier transform of the function
$$\widehat{\psi}(y)=\overline{m_0\left(\pi+
y/2\right)}\cdot\exp\left(-iy/2\right)\cdot\widehat{\phi}\left(y/2\right).$$ Then the
function
$$\psi(x)=\frac1{2\pi}\int_{\mathbb
R}e^{iyx}{\widehat{\psi}(y)}\,dy$$ is called the $m$-wavelet.

Let
\begin{equation}\label{1}\phi_{jk}(x)=2^{j/2}\phi(2^jx-k),\quad
\psi_{jk}(x)=2^{j/2}\psi(2^jx-k),\quad j,k \in\mathbb Z\,.
\end{equation}
It is known that the family of functions $\{\phi_{0k},
\psi_{jk}:\,j\in \mathbb N_0,\ k\in \mathbb Z\},$ $\mathbb N_0:=\{0,1,2,...\},$ is an orthonormal basis in
$L_2(\mathbb R)$ (see, for example, \cite{chu,dau}).

An arbitrary function $f(x)\in L_2(\mathbb R)$ can be represented in
the form
\begin{equation}\label{2.5}f(x)=\sum_{k\in\mathbb Z}\alpha_{0k}\phi_{0k}(x)+\sum_{j=0}^{\infty}\sum_{k\in\mathbb Z}\beta_{jk}\psi_{jk}(x)\,,
\end{equation}
$$\alpha_{0k}=\int_{\mathbb R}f(x)\overline{\phi_{0k}(x)}\,dx,\quad \beta_{jk}=\int_{\mathbb R}f(x)\overline{\psi_{jk}(x)}\,dx.$$
The representation (\ref{2.5}) is called a wavelet representation.

The series~(\ref{2.5}) converges in the space $L_2(\mathbb R),$ i.e.
$\sum_{k\in\mathbb
Z}|\alpha_{0k}|^2+\sum_{j=0}^{\infty}\sum_{k\in\mathbb
Z}|\beta_{jk}|^2<\infty\,.$

\

\noindent {\bf Assumption S.} \cite{har}  For the $f$-wavelet $\phi$
there exists a decreasing function $\Phi(x),$ $x\ge 0,$ such that
$\Phi(0)<\infty,$ $|\phi(x)|\le
\Phi(|x|)$ (a.e.) and $\int_{\mathbb R}\Phi(|x|)\,dx<\infty\,.$

\begin{thm}\label{KozPerNerPsi} {\rm\cite{kozperUMJ}} If the assumption {\rm S}  holds  for the $f$-wavelet
$\phi$ and a function $\Phi(\cdot),$ then there exists such $B>0$ that for all $x\in \mathbb R:$
$$
 |\psi(x)|\le B \cdot \Phi \left(\left|
  \frac{2x-1}4\right| \right).$$  
\end{thm}

\begin{cor}\label{KozPerCompPsi} If the $f$-wavelet
$\phi$ is a bounded function with a compact support, then
the corresponding $m$-wavelet $\psi$ is also bounded and has a compact support.
\end{cor}

\begin{proof} Consider $f$-wavelet $\phi$ which support is in the interval $[-a, a].$ Then the assumption {\rm S}  holds true for the
$f$-wavelet $\phi$  with the function
$$\Phi_\phi(x):=    \begin{cases}\sup_{x \in [-a, a]}|\phi(x)|, & x \in [-a, a],\\
0, & x \not\in [-a, a].
\end{cases}$$
Using Theorem~\ref{KozPerNerPsi} we obtain
$ |\psi(x)|\le B \cdot \Phi_\phi\left(\left|{2x-1}\right|/4 \right),$ $x\in \mathbb R.$
By this inequality $\psi(x)=0,$ when $x \not\in \left[\frac{1-4a}2,
\frac{1+4a}2\right].$ Hence, $\psi$ also has a compact support.
\end{proof}

\begin{rmk}\label{rema1} If the $f$-wavelet
$\phi$ is a bounded function with a compact support in $[-a, a],$
then the assumption {\rm S}  holds true for the $m$-wavelet $\psi$ with a certain function
$\Phi_\psi(\cdot),$ such that for some $\hat a \ge 0:$ $\Phi_\psi(x)=0,$ when $x
\not\in \left[-\hat a, \hat a\right].$ As an example of such function
we can take
$$\Phi_\psi(x):=
\begin{cases}
\sup_{x \in [-\hat a, \hat a]}|\psi(x)|, & x \in [-\hat a, \hat a],\\
0, & x \not\in [-\hat a, \hat a],
\end{cases},\quad \hat a=\frac{1+4a}2.$$
\end{rmk}

\begin{thm}\label{UnifConvFunc} Let the assumption {\rm S}  hold true for the $f$-wavelet
$\phi$ with a function $\Phi(\cdot).$
Let $f(x),$ $x\in\mathbb R,$ be a measurable continuous on
$(a, b),$ $a,b\in\mathbb R,$ function. Suppose that  $|f(x)|\le
c(x), \ x\in \mathbb R,$ where $c(\cdot)$ is an even increasing on
$[0,\infty)$  function such that $c(x)>1$ and
\begin{equation}\label{intcfi}\int_{\mathbb
R}c(ax)\Phi(|x|)\,dx<\infty,\ \mbox{for all}\ a>0.\end{equation}

Then
\begin{equation}\label{fmtof}f_n(x)=\sum_{k\in\mathbb
Z}\alpha_{0k}\phi_{0k}(x)+\sum_{j=0}^{n-1}\sum_{k\in\mathbb
Z}\beta_{jk}\psi_{jk}(x)\to f(x),\ n\to\infty,\end{equation}
uniformly on each interval $ [\alpha, \ \beta]\subset(a, b).$ 
\end{thm}
This theorem is a modification of  Theorem 6.1 in {\rm\cite{kozperUMJ}}.

\begin{cor}\label{unconfcomsup} Let the $f$-wavelet
$\phi$ be a bounded function with a compact support and $f(x), x\in\mathbb R,$ be a
continuous function on $\mathbb R.$ Then $f_n(x)\to f(x)$ uniformly on each interval $ [\alpha, \ \beta],$ when
$n\to\infty.$
\end{cor}

\begin{proof} The statement of the corollary
follows from Corollary~\ref{KozPerCompPsi} and Theorem~\ref{UnifConvFunc}. The assumption {\rm S} holds true for the $f$-wavelet
$\phi$ with the function
$\Phi_\phi(\cdot)$ (which has a compact support) given in Corollary~\ref{KozPerCompPsi}. Therefore, condition
(\ref{intcfi})  holds true. An application of Theorem~\ref{UnifConvFunc} completes the proof.
\end{proof}
\section{Wavelet representation of random processes}
Let $\{\Omega, \cal{B}, P\}$ be a standard probability space. Let $\mathbf{X}(t),$ $t\in\mathbb R$ be a random process such that
$\mathbf E\mathbf{X}(t)=0\,,$ $\mathbf E|\mathbf{X}(t)|^2<\infty$ for all $t\in\mathbb R.$

We want to construct a representation of the kind~(\ref{2.5}) for
$\mathbf{X}(t)$ with mean-square integrals
$$\xi_{0k}=\int_{\mathbb R}
\mathbf{X}(t)\overline{\phi_{0k}(t)}\,dt,\quad
\eta_{jk}=\int_{\mathbb R}\mathbf{X}(t)\overline{\psi_{jk}(t)}\,dt\,.$$
Consider the approximants $\mathbf{X}_{n}(t)$
of $\mathbf{X}(t)$
defined by
\begin{equation}\label{7}\mathbf{X}_{n}(t)
:=\sum_{k\in \mathbb Z}\xi_{0k}\phi_{0k}(t)+\sum_{j=0}^{n-1}
\sum_{k\in \mathbb Z}\eta_{jk}\psi_{jk}(t)\,.
\end{equation}
In applied wavelet analysis $\mathbf{X}_{n}(t)$ is interpreted as an approximation of $\mathbf{X}(t)$ by a smooth component (the first sum) and $n$ details (the second sum), see \cite{wmtsa}.

Theorem~\ref{TheorMeanSquareComp} below guarantees
 the mean-square
convergence of $\mathbf{X}_{n}(t)$ to $\mathbf{X}(t)$
 in the case when the $f$-wavelet $\phi$ has
a compact support. \vspace{1mm}

\begin{thm}\label{TheorMeanSquareComp}
 Let $\mathbf{X}(t),$ $t\in\mathbb R,$ be
 a random process which covariance function  $R(t,s)$  is continuous on $\mathbb{R}^2.$ Let the $f$-wavelet
$\phi$ and the $m$-wavelet $\psi$ be
 continuous functions and the
$f$-wavelet $\phi$ have
a compact support. Then
 $\mathbf{X}_{n}(t)\to \mathbf{X}(t)$
  in mean square when
   $n\to\infty.$
\end{thm}
\begin{proof} By Corollary~\ref{KozPerCompPsi} the $m$-wavelet $\psi$ has
a compact support.
 Let us show that the integrals, which define
 $\xi_{0k}$ and $\eta_{jk},$ exist and belong to
$L_2(\Omega).$ We only consider the integral
$\eta_{jk}=\int_{\mathbb
R}\mathbf{X}(t)\overline{\psi_{jk}(t)}\,dt$  as for  $\xi_{0k}$ the analysis is similar. First we show that $\eta_{jk}$ has a finite second moment. By the definition of $\eta_{jk}$ we get
\begin{equation}\label{e2}\mathbf E|\eta_{jk}|^2=\int_{\mathbb R}\int_{\mathbb R}\mathbf E
\mathbf{X}(t)\overline{\mathbf{X}(s)}
\overline{\psi_{jk}(t)}\psi_{jk}(s)\,dtds= \int_{\mathbb
R}\int_{\mathbb R}R(t,s) \overline{\psi_{jk}(t)}\psi_{jk}(s)\,dtds.
\end{equation}
The last integral in (\ref{e2}) exists, because $R(t,s)$ and $\psi_{jk}(t)$ are continuous and $\psi_{jk}(t)$ has
a compact support. Therefore $\eta_{jk}$ is correctly defined. 

Note that for a fixed $t$
the sums $\sum_{k\in \mathbb Z}\xi_{0k}\phi_{0k}(t)$ and $\sum_{k\in
\mathbb Z}\eta_{jk}\psi_{jk}(t)$ have finite numbers of terms, because, by (\ref{1}), $\phi_{0k}(t)$ and $\psi_{jk}(t)$ vanish for sufficiently large $k.$

Now we show that $\mathbf E|\mathbf X_n(t)-\mathbf X(t)|^2\to 0$, when
$n\to\infty$. If we rewrite
$$\mathbf E|\mathbf X_n(t)-\mathbf X(t)|^2=\mathbf E|\mathbf X(t)|^2
-\mathbf E\mathbf{X}(t)\overline{\mathbf X_n(t)}-\mathbf
E\overline{\mathbf X(t)}\mathbf X_n(t)+\mathbf E|\mathbf
X_n(t)|^2,$$ then it is enough to prove that
  $\mathbf
E\mathbf{X}(t)\overline{\mathbf{X}_n(t)}\to R(t,t)=\mathbf
E|\mathbf{X}(t)|^2$ and
 $\mathbf E|\mathbf{X}_n(t)|^2\to R(t,t),$ when  $n\to\infty$.

It is easy to see that
$$\mathbf
E\mathbf{X}(t)\overline{\mathbf{X}_n(v)}=\sum_{k\in \mathbb
Z}\int_{\mathbb R}R(t,u)\overline{\phi_{0k}(u)}\,du\cdot\phi_{0k}(v)+
\sum_{j=0}^{n-1}\sum_{k\in \mathbb Z}\int_{\mathbb
R}R(t,u)\overline{\psi_{jk}(u)}\,du\cdot\psi_{jk}(v)=:R_n(t,v).$$

We consider sum (\ref{fmtof}) where instead of the function $f_n(\cdot)$ we
use $R_n(t,\cdot).$ Note, that $R_n(t,\cdot)$ is a continuous function of the second argument. By Corollary~\ref{unconfcomsup} for each fixed $t$  we get $R_n(t,v)\to R(t,v)$
uniformly in $v$ on each interval $ [\alpha, \
\beta],$ when $n\to\infty.$

By (\ref{7}) we obtain
$$\mathbf
E|\mathbf{X}_n(t)|^2=\sum_{k\in \mathbb Z}\int_{\mathbb R}\mathbf
E\mathbf{X}(t)\overline{\mathbf{X}_n(u)\phi_{0k}(u)}\,du\cdot\phi_{0k}(t)+
\sum_{j=0}^{n-1}\sum_{k\in \mathbb Z}\int_{\mathbb R}\mathbf
E\mathbf{X}(t)\overline{\mathbf{X}_n(u)\psi_{jk}(u)}\,du\cdot\psi_{jk}(t).$$
There is a finite number of non-zero terms in the sums above.
Using $\mathbf E\mathbf{X}(t)\overline{\mathbf{X}_n(u)}\to R(t,u)$
uniformly in $u$ on each interval, when $n\to\infty,$ we conclude that
$$\mathbf
E|\mathbf{X}_n(t)|^2\to\sum_{k\in \mathbb Z}\int_{\mathbb
R}R(t,u)\overline{\phi_{0k}(u)}\,du\cdot\phi_{0k}(t)+
\sum_{j=0}^{n-1}\sum_{k\in \mathbb Z}\int_{\mathbb
R}R(t,u)\overline{\psi_{jk}(u)}\,du\cdot\psi_{jk}(t)=R(t,t).$$
\end{proof}

\section{Uniform convergence  of wavelet expansions}
The sequence $\mathbf{X}_{n}(t)$ converges to $\mathbf{X}(t)$ uniformly in probability
on each interval $[0,T]$ (in probability in Banach space $C([0,T]),$ see \cite{bulkoz}), if
$$P\left\{\sup_{0\le t\le T} |\mathbf{X}(t)-\mathbf{X}_{n}(t)|>\varepsilon \right\}\to 0,\quad n\to\infty.
$$

 The following theorem is the first main
result of the paper.
\begin{thm}\label{mainuniformproccompnos}
 Let $\mathbf{X}(t),$ $t\in\mathbb R,$ be
 a separable Gaussian random process
such that $\mathbf E\mathbf{X}(t)=0,$ $\mathbf
E|\mathbf{X}(t)|^2<\infty$ for all $t\in\mathbb R,$ and its
covariance function  $R(t,s)$  is continuous on $ \mathbb R^2.$ Let the $f$-wavelet
$\phi$ and the $m$-wavelet $\psi$ be
 continuous functions and the
$f$-wavelet $\phi$ have
a compact support. Suppose that for some
$\alpha>1/{2}:$
\begin{enumerate}
\item[{\rm (i)}] the  integral $\int_{\mathbb R}\ln^{\alpha}(1+|u|)|\widehat{\psi}(u)|\,du$ converges;
\item[{\rm (ii)}] there exist constants $\{c_j,\,j\in \mathbb N_0\},$ such that
 $|\mathbf E\eta_{jk}\overline{\eta_{jl}}|\le c_j$ for all $j\in \mathbb N_0,$ and 
\begin{equation}\label{sumcj}\sum\limits
_{j=1}^{\infty}\sqrt{c_j}\, 2^{\frac
j2}j^{\alpha}<\infty.\end{equation} 
\end{enumerate}

Then $\mathbf{X}_{n}(t)\to \mathbf{X}(t)$ uniformly in probability
on each interval $[0,T]$ when  $n\to\infty.$
\end{thm}

\section{Proof of the main theorem}
To prove Theorem~\ref{mainuniformproccompnos}
 we need some auxiliary
results.

\begin{thm}\label{kozsliconv}{\rm \cite{kozol1}} Let $\mathbb T =[0,T]$, $\rho(t,s)=|t-s|$. Let $\mathbf{X}_n(t),$ $t\in\mathbb T,$
 be a sequence of Gaussian stochastic processes such that
all  $\mathbf{X}_n(t)$ are separable in $(\mathbb T, \rho)$ and
$$\sup_{n\ge 1}\sup_{|t-s|\le h}\left( \mathbf E|\mathbf{X}_n(t)-\mathbf{X}_n(s)|^2\right)^{1/2}\le \sigma(h)\,,$$
where $\sigma(h)$ is a function, which is monotone increasing in a
neighborhood of the origin and $\sigma(h)\to 0$ when $h\to 0\,.$

Assume that for some $\varepsilon>0$
\begin{equation}\label{8}\int_0^\varepsilon
\sqrt{-\ln\left(\sigma^{(-1)}(h)\right)}\,dh<\infty\,,\end{equation}
where $\sigma^{(-1)}(\cdot)$ is the inverse function of $\sigma(\cdot)$.

 If the processes $\mathbf{X}_n(t)$ converge in probability to the process
 $\mathbf{X}(t)$ for all $t\in \mathbb T$, then $\mathbf{X}_n(t)$ converges  to  $\mathbf{X}(t)$ in the space
$C(\mathbb T).$
\end{thm}
\begin{rmk}\label{sigma} It is easy to check that {\rm (\ref{8})} holds true for
$\sigma(h)=\frac C{\left(\ln\left(e^{\alpha}+ \frac1{h}\right)\right)^{\beta}}$ and $\sigma(h)=C h^{\kappa}$  when $C>0,$ $\beta>1/2,$ $\alpha>0,$ and $\kappa>0.$
\end{rmk}

\begin{lem}\label{lem1NerLgamma}  Let $[-\hat a, \hat a]$ be
a compact support of the function $\Phi_\psi(\cdot),$
$$S_{\gamma}(x):=\sum\limits_{k\in\mathbb Z}
 |\psi(x-k)|^{\gamma}, \ 0<\gamma\le1.$$
Then $\sup\limits_{x\in\mathbb R}S_{\gamma}(x)\le L_{\gamma},$
where \begin{equation}\label{Lgamma}L_{\gamma}=
\begin{cases}
3\Phi_\psi^\gamma(0)+4\int_{1/2}^{\hat a}\Phi_\psi^{\gamma}(t)\,dt,
 & \hat a>\frac12,\\
3\Phi_\psi^\gamma(0), & \hat a<\frac12. \end{cases}\end{equation}
\end{lem}
\begin{proof}
Since $S_{\gamma}(x)$ is a periodic function with period 1,  it is
sufficient to prove the assertion of the lemma for $x\in[0,1].$

Notice, that for $x\in[0,1]$ and integer $|k|\ge 2$ the inequality
$|x-k|\ge |k|/2$ holds true.  Hence, $\Phi_\psi(|x-k|)\ge
\Phi_\psi\left(|k|/2\right)$ and for $\hat a>\frac12$
$$S_{\gamma}(x)\le \Phi_\psi^{\gamma}(|x|)+\Phi_\psi^{\gamma}(|x+1|)
+\Phi_\psi^{\gamma}(|x-1|)+\sum\limits_{|k|\ge2}\Phi_\psi^{\gamma}\left({|k|}/2\right)$$
$$\le
3\Phi_\psi^{\gamma}(0)+2\sum\limits_{k=2}^{\infty}\,
\int_{k-1}^{k}\Phi_\psi^{\gamma}\left(
t/2\right)\,dt=3\Phi_\psi^{\gamma}(0)+4\int_{\frac
12}^{\hat a}\Phi_\psi^{\gamma}(t)\,dt.$$ For $\hat a\le1$ the statement
is obvious.
\end{proof}

\begin{lem}\label{lem2Nerphiln}  If $\int_{\mathbb R}\ln^{\alpha}(1+|u|)|\widehat{\psi}(u)|\,du<\infty,
$ for some $\alpha>0,$
then for all $x,y\in \mathbb R,$ $j\in \mathbb N_0,$ and
$k\in\mathbb Z:$
$$
|\psi_{jk}(x)-\psi_{jk}(y)|\le 2^{\frac j2}\max(1,j^{\alpha})R_{\alpha}\cdot
\ln^{-\alpha}\left(e^{\alpha}+\frac{1}{|x-y|}\right),
$$ where
 $R_{\alpha}:=\frac{1}{\pi}\int_{\mathbb
R}\ln^{\alpha}\left(e^{\alpha}+|u|+2\right)|\widehat{\psi}
(u)|\,du.$
\end{lem}
\begin{proof}
Since $\psi(x)=\frac{1}{2\pi}\int_{\mathbb
R}e^{ixu}\hat\psi(u)\,du,$ we have
\begin{equation}\label{intsin}|\psi_{jk}(x)-\psi_{jk}(y)|
\le\frac{2^{\frac j2}}{2\pi}\int_{\mathbb R}
\left|e^{ixu2^j}-e^{iyu2^j}\right||\hat\psi\left(
u\right)|\,du=\frac{2^{\frac
j2}}{\pi}\int_{\mathbb R}
\left|\sin((x-y)u2^{j-1})\right||\hat\psi\left( u\right)|\,du.
\end{equation}
Note that for $v\ne0$ the following inequality holds (see the inequality (8) in \cite{kozol1}):
\begin{eqnarray}\label{sinln}\left|\sin\left(\frac{s}v\right)\right|\le
\frac{\left|\ln\left(e^{\alpha}+|s|\right)\right|^{\alpha}}
{\left|\ln\left(e^{\alpha}+|v|\right)\right|^{\alpha}}\,.
\end{eqnarray}
By~(\ref{intsin}) and~(\ref{sinln}) we obtain
\begin{eqnarray*}|\psi_{jk}(x)-\psi_{jk}(y)|\le
\frac{2^{\frac
j2}}{\pi\ln^{\alpha}\left(e^{\alpha}+\frac1{|x-y|}\right)}\int_{\mathbb
R} {\ln^{\alpha}\left(e^{\alpha}+|u|2^{j-1}\right)}|\hat\psi\left( u\right)|\,du.
\end{eqnarray*}
The assertion of the lemma follows from the inequality
\begin{eqnarray*}e^{\alpha}+\frac
{|u|2^j}2\le  e^{\alpha
{\max(1,j)}}+2^{\max(1,j)}\left(\frac {|u|+2}2\right)^{\max(1,j)}\le \left(e^{\alpha}+
{|u|+2}\right)^{\max(1,j)},\ j\in \mathbb N_0.
\end{eqnarray*}
\end{proof}

\begin{lem}\label{lemsumpsi} If the $m$-wavelet $\psi(x)$
satisfies the assumptions of Lemmata~{\rm\ref{lem1NerLgamma}}
 and~{\rm\ref{lem2Nerphiln}}, then for all $\gamma\in[0,1),$
$\alpha>0,$ $x,y\in \mathbb R,$ $j\in \mathbb N_0,$ and
$k\in\mathbb Z$ the following inequality holds true
\begin{equation}\label{sumpsijk-}\sum\limits_{k\in\mathbb
Z}|\psi_{jk}(x)-\psi_{jk}(y)|\le
2^{\frac{j}{2}+1}R^{1-\gamma}_\alpha
L_{\gamma}\,{\max(1,j^{\alpha(1-\gamma)})}\ln^{-\alpha(1-\gamma)}\left(e^{\alpha}+\frac{1}{|x-y|}\right).
\end{equation}
\end{lem}
\begin{proof}
By Lemma~\ref{lem2Nerphiln} we obtain
\begin{eqnarray}\label{sumpsijk-1}\sum\limits_{k\in\mathbb
Z}|\psi_{jk}(x)-\psi_{jk}(y)|=\sum\limits_{k\in\mathbb
Z}|\psi_{jk}(x)-\psi_{jk}(y)|^{1-\gamma}
\cdot|\psi_{jk}(x)-\psi_{jk}(y)|^{\gamma}\nonumber\\
\le \ln^{-\alpha(1-\gamma)}\left(e^{\alpha}+\frac1{|x-y|}\right)
\left(2^{\frac{j}{2}}
{\max(1,j^{\alpha})}R_\alpha\right)^{1-\gamma}\cdot\sum\limits_{k\in\mathbb
Z}|\psi_{jk}(x)-\psi_{jk}(y)|^{\gamma}.
\end{eqnarray}
By Lemma~\ref{lem1NerLgamma} we get
\begin{equation}\label{sg}\sum\limits_{k\in\mathbb
Z}|\psi_{jk}(x)-\psi_{jk}(y)|^{\gamma}\le
2\sup\limits_{x\in\mathbb R}\sum\limits_{k\in\mathbb
Z}|\psi_{jk}(x)|^{\gamma}=2^{\frac{j\gamma}{2}+1}
\sup\limits_{x\in\mathbb R}\sum\limits_{k\in\mathbb
Z}|\psi(2^jx-k)|^{\gamma}\le  2^{\frac{j\gamma}{2}+1}L_{\gamma}.
\end{equation}
Finally, the upper bound in (\ref{sumpsijk-}) follows from  inequalities (\ref{sumpsijk-1})  and
(\ref{sg}).
\end{proof}
Now we are ready to prove Theorem~\ref{mainuniformproccompnos}.

\begin{proof}
In virtue of Theorem~\ref{kozsliconv} and Remark~\ref{sigma}
  it is
sufficient to prove that $\mathbf{X}_{n}(t)\to \mathbf{X}(t)$ in
mean square and for some $C>0,$ $\beta>1/2,$ and $\alpha>0:$
\begin{eqnarray}\label{nerEXmln}\left(\mathbf{E}\left|\widehat{\mathbf{X}}_{n}(t)
-\widehat{\mathbf{X}}_{n}(s)\right|^2\right)^{\frac12}\le \frac
C{\ln^{\beta}\left(e^{\alpha}+\frac1{|t-s|}\right)},\end{eqnarray}
where
$\widehat{\mathbf{X}}_{n}(t)=
\sum\limits_{j=0}^{n-1}\sum\limits_{k\in\mathbb
Z}\eta_{jk}\psi_{jk}(t).$

We use the fact that the convergence doesn't depend on the first summand in (\ref{7}). Note, that $\widehat{\mathbf{X}}_{n}(\cdot)$ satisfies the conditions of Theorem~\ref{TheorMeanSquareComp}.
Hence, by Theorem~\ref{TheorMeanSquareComp} we get
$\widehat{\mathbf{X}}_{n}(t)\to \mathbf{X}(t)$ in mean square, when $n\to \infty.$

To get (\ref{nerEXmln}) we use the estimate
\begin{eqnarray*}\label{}\left(\mathbf{E}\left|\widehat{
\mathbf{X}}_{n}(t)-
\widehat{\mathbf{X}}_{n}(s)\right|^2\right)^{\frac12} \le
\sum\limits_{j=0}^{n-1}\left(\mathbf E\left|\sum\limits_{k\in\mathbb
Z}\eta_{jk}(\psi_{jk}(t)-\psi_{jk}(s))\right|^2\right)^{1/2}=:
\sum\limits_{j=0}^{n-1}\sqrt{S_j}\,.
\end{eqnarray*}
$S_j$ can be bounded as follows
$$S_j=\sum\limits_{k\in\mathbb Z}\sum\limits_{k\in\mathbb Z}
 \mathbf
 E\eta_{jk}\overline{\eta_{jl}}(\psi_{jk}(t)-\psi_{jk}(s))(\overline{\psi_{jl}(t)}-\overline{\psi_{jl}(s)})
\le c_{j}\left(\sum\limits_{k\in\mathbb Z}
|\psi_{jk}(t)-\psi_{jk}(s)|\right)^2. $$

By Lemma~\ref{lemsumpsi}
$$S_j\le 2^jc_jL^2\,\max(1,j^{2\alpha(1-\gamma)})\ln^{-2\alpha(1-\gamma)}\left(e^{\alpha}+\frac1{|t-s|}\right),$$
where  $\gamma\in (0,1),$ $L:=2R^{1-\gamma}_\alpha L_\gamma<\infty.$

Therefore, we conclude that 
\begin{eqnarray}\label{hatXmNer}\left(\mathbf{E}\left|\widehat{\mathbf{X}}_{n}(t)-
\widehat{\mathbf{X}}_{n}(s)\right|^2\right)^{\frac12}\le \frac
C{\ln^{\alpha(1-\gamma)}\left(e^{\alpha}+\frac1{|t-s|}
\right)},\end{eqnarray} where by condition (ii) of the theorem
$$C:=L\left(\sqrt{c_0}+\sum\limits_{j=1}^{\infty}\sqrt{c_j}2^{\frac
j2}j^{\alpha(1-\gamma)}\right)<\infty.$$

By taking $\alpha>1/2,$ $\gamma$ sufficiently close to 0,  and $\beta=\alpha (1-\gamma),$ we obtain $\beta>1/2,$ which completes the proof of the theorem.
\end{proof}

\section{Convergence rate in the space $C[0, T]$}
In this section we investigate what happens when the number of terms in the approximant~(\ref{7}) becomes large.

\begin{thm}\label{71} Let $\mathbf
Y(t),$ $t\in[0,T]$ be a separable Gaussian random process,
$$\varepsilon_0:=\sup_{0\le t\le T}\left(\mathbf
\mathbf E|\mathbf Y(t)|^2\right)^{1/2}<\infty\,,$$ and
$$\sup\limits_{\scriptsize\begin{array}{c}
                     |t-s|\le \varepsilon\\
                      t,s\in [0, T]
                   \end{array}}\left(\mathbf
                   E|\mathbf
Y(t)-\mathbf Y(s)|^2\right)^{1/2}\le\frac
C{\ln^{\beta}\left(e^{\alpha}+ \frac1{\varepsilon}\right)}=:\sigma(\varepsilon),  \quad
\beta>1/2,\ \alpha>0\,.
$$
Then for $u>8\delta(\varepsilon_{0})$
$$P\left\{\sup_{0\le t\le T} |\mathbf
Y(t)|>u \right\}\le
2\exp\left\{-\frac{\left(u-\sqrt{8u\delta(\varepsilon_0)}\right)^2}
{2\varepsilon_0^2}\right\}\,,$$
where $\nu:=\min \left(\varepsilon_0, \sigma\left(T/2\right)\right)$ and
$$\delta(\varepsilon_0):
=\frac{\nu}{\sqrt2}\left(\sqrt{\ln(T+1)}+
\left({1-\frac{1}{2\beta}}\right)^{-1}
\left(\frac{C}{\nu}\right)^{\frac{1}{2\beta}}\right).$$
\end{thm}
This result follows from Lemmata 1 and 2 in \cite{kozol2}.

\begin{thm} Let  a separable Gaussian random process
 $\mathbf{X}(t),$ $t\in [0,T],$  the $f$-wavelet $\phi,$
and the corresponding $m$-wavelet $\psi$  satisfy the assumptions of
Theorem~{\rm\ref{mainuniformproccompnos}} for some $\alpha>1/2.$

Then
$$P\left\{\sup_{0\le t\le T} |\mathbf{X}(t)-\mathbf{X}_{n}(t)|>u
 \right\}\le
2\exp\left\{-\frac{(u-\sqrt{8u\delta(\varepsilon_{n})})^2}
{2\varepsilon_{n}^2}\right\}\,,\quad u>8\delta(\varepsilon_{n}),\ n\ge 1,$$ where  $\delta(\cdot),$ $\nu,$ $\sigma,$ $\beta\in (1/2,\alpha),$ and $C=C_n$ are defined in Theorem~{\rm\ref{71}}.  The decreasing sequences $\varepsilon_{n}$ and $C_n$ are determined by {\rm(\ref{C_n})} and
{\rm(\ref{epsilon_n})} respectively in the proof of the theorem.
\end{thm}
\begin{proof}
Note that
$\Delta_n(t):=\mathbf{X}(t)-\mathbf{X}_{n}(t)=
\sum\limits_{j=n}^{\infty}\sum\limits_{k\in\mathbb Z}
 \eta_{jk}\psi_{jk}(t).$

Similarly to (\ref{hatXmNer}) we  get the inequality
$$
\left(\mathbf{E}\left|\Delta_{n}(t)-
\Delta_n(s)\right|^2\right)^{\frac12}\le \frac
{C_n}{\ln^{\alpha(1-\gamma)}\left(e^{\alpha}+\frac1{|t-s|}
\right)},
$$
where $\alpha>1/2,$ $0<\gamma<1,$
$C_n:=L
\cdot\sum_{j=n}^\infty\sqrt{c_j}\,2^ {j}j^{\alpha(1-\gamma)}.$

By taking  $\beta=\alpha (1-\gamma)$ we can rewrite $C_n$ in terms of $\alpha$ and $\beta$ in the following manner
\begin{equation}\label{C_n}
C_n=2R_\alpha^{\beta/\alpha}L_{1-\beta/\alpha}\cdot\sum_{j=n}^\infty\sqrt{c_j}\,2^ {j}j^{\beta}<\infty.
\end{equation}

Then
\begin{equation}\label{e1}\left(\mathbf
E|\Delta_n(t)|^2\right)^{1/2}=\left(\mathbf
E\left|\sum_{j=n}^{\infty} \sum_{k\in\mathbb
Z}\eta_{jk}\psi_{jk}(t)\right| ^2\right)^{1/2}\le
\sum_{j=n}^{\infty}\left(\mathbf E\left| \sum_{k\in\mathbb
Z}\eta_{jk}\psi_{jk}(t)\right| ^2\right)^{1/2}.\end{equation}

By Lemma~\ref{lem1NerLgamma} we get
\begin{equation}\label{e11}\mathbf E\left| \sum_{k\in\mathbb
Z}\eta_{jk}\psi_{jk}(t)\right|^2\le c_j\left( \sum_{k\in\mathbb
Z}|\psi_{jk}(t)|\right)^2=2^jc_j\left( \sum_{k\in\mathbb
Z}|\psi(2^jt-k)|\right)^2\le 2^j\, c_j\, L_1^2,\end{equation}
where $L_1$ is defined by (\ref{Lgamma}).

Hence, by the assumptions of the theorem, (\ref{e1}), and (\ref{e11}) we obtain
 \begin{eqnarray}\label{epsilon_n}\varepsilon_n:=\sup\limits_{0\le t\le1}\left(\mathbf
E|\Delta_n(t)|^2\right)^{1/2}\le L_1\sum_{j=n}^{\infty}
\sqrt{c_{j}}\,2^{j/2}<\infty. \end{eqnarray}

By taking $\gamma$ sufficiently close to 0 we obtain $\beta>1/2.$
Thus, the statement of the theorem follows from Theorem~{\rm\ref{71}}.
 \end{proof}

\section{Examples}\label{sec7}
In this section we provide examples of wavelets and stationary stochastic processes which satisfy assumption~(\ref{sumcj}) of
Theorem~\ref{mainuniformproccompnos}. Note that the assumptions are simpler than those used in results for stationary processes in \cite{kozol1}.

\noindent\textbf{Example 1}. Let the $f$-wavelet $\phi$ and the $m$-wavelet $\psi$ be  continuous functions and $\phi(\cdot)$ has the compact support $[-a,a].$
Assume that  $\widehat{\psi}$ is a Lipschitz function of order $\varkappa>0,$ i.e.
\begin{equation}\label{lip}|\widehat{\psi}(u)-\widehat{\psi}(v)|\le C|u-v|^{\varkappa}.\end{equation}
 
Let  $\mathbf{X}(t)$ be a centered  stationary stochastic process which covariance function $R(t-s):=\mathbf E\mathbf{X}(t)\overline{\mathbf{X}(s)}$ corresponds to the spectral density $\widehat{R}(\cdot)$  satisfying the following condition
\begin{equation}\label{intkap}\int_{\mathbb R}\left|\widehat{R}(z)\right|\cdot\left|z\right|^{\varkappa}\,dz<\infty.\end{equation}

Using (\ref{1}), the identity
$\label{hatpsijk}\widehat{\psi}_{jk}(z)={2^{-j/2}}{e^{-i\frac k{2^j}z}}\,\widehat{\psi}\left(z/{2^j}\right),$
and Parseval's theorem we obtain
$$|\mathbf E\eta_{jk}\overline{\eta_{jl}}|=\left|\int_{\mathbb R}\int_{\mathbb R}R(u-v)\overline{\psi_{jk}(u)}\,du\,\psi_{jl}(v)\,dv\right|=\left|\int_{\mathbb R}\int_{\mathbb R}\frac{e^{-ivz}}{2 \pi}\widehat{R}(z)\overline{\widehat{\psi}_{jk}(z)}\,dz\,\psi_{jl}(v)\,dv\right|$$
$$\le\frac{2^{j/2}}{2 \pi}\int_{\mathbb R}\left|\widehat{R}(z)\right|\,\left|\overline{\widehat{\psi}_{jk}(z)}\right|\,dz\cdot \int_{\mathbb R}\left|\psi(2^jv-l)\right|\,dv\le \frac{\hat{a}\,\Phi_\psi(0)}{2^{j}
\pi}\int_{\mathbb
R}\left|\widehat{R}(z)\right|\cdot\left|\widehat{\psi}\left(\frac
z{2^j}\right)\right|\,dz,$$
where $\hat{a}$ and $\Phi_\psi(\cdot)$ are given in Remark~\ref{rema1}.

Note that $\widehat{\psi}(0)=0$ by properties of the $m$-wavelet $\psi.$  Hence, by (\ref{lip}) and (\ref{intkap})
$$\left|\mathbf
E\eta_{jk}\overline{\eta_{jl}}\right|\le\frac{\hat{a}\,\Phi_\psi(0)}{\pi\,
2^{j(1+\varkappa)}}\int_{\mathbb
R}\left|\widehat{R}(z)\right|\cdot\left|z\right|^{\varkappa}\,dz,\quad \mbox{for all}\ k,l\in\mathbb Z.$$
Therefore $\sqrt{c_j}\le C/2^{j(1+\varkappa)/2}$ and
assumption~(\ref{sumcj}) is satisfied.

\begin{rmk} By Remark~{\rm\ref{rema1}}, $\psi(\cdot)$ has the compact support $[-\hat{a},\hat{a}]$ and is bounded. Hence,
$$|\widehat{\psi}(u)-\widehat{\psi}(v)|\le \int_{-\hat{a}}^{\hat{a}}\left|e^{-iut}-e^{-ivt}\right||\psi(t)|\,dt\le \Phi_\psi(0)\int_{-\hat{a}}^{\hat{a}}\left|\sin\left(\frac{(u-v)t}{2}\right)\right|\,dt\le \frac{\hat{a}^2\Phi_\psi(0)}{2}|u-v|.$$ 
Therefore,  $\widehat{\psi}$ is always a Lipschitz function (at least of order 1).
\end{rmk}

\noindent\textbf{Example 2}. In this numerical example we use the Daubechies D8 wavelet and the process $\mathbf{X}(t),$ $t\in[0,30],$ with the covariance function $R(t)=e^{-t^2}.$ Realizations of the process and two wavelet reconstructions of a realization are plotted in Figures~1 and~2.  The figures have been generated by the R packages \textsc{geoR} \cite{geoR} and \textsc{wmtsa} \cite{wmtsa}.  Figure~3 shows the maximum absolute approximation errors $\sup_{0\le t\le 30} |\mathbf{X}(t)-\mathbf{X}_{n}(t)|,$ $n=1,2,3,$ for 500 simulated realizations of $\mathbf{X}(t)$ for each $n$. Finally, the mean reconstruction error as a function of the number $n$ of detail terms is given in Figure~4.  We clearly see that empirical probabilities of large errors become smaller when the number of terms in the wavelet expansion increases.
\noindent\begin{figure}
\begin{minipage}{7.5cm}
 \includegraphics[width= 7.5cm,height=6.1cm,trim=0cm 0.5cm 0.5cm 1.6cm, clip=true]{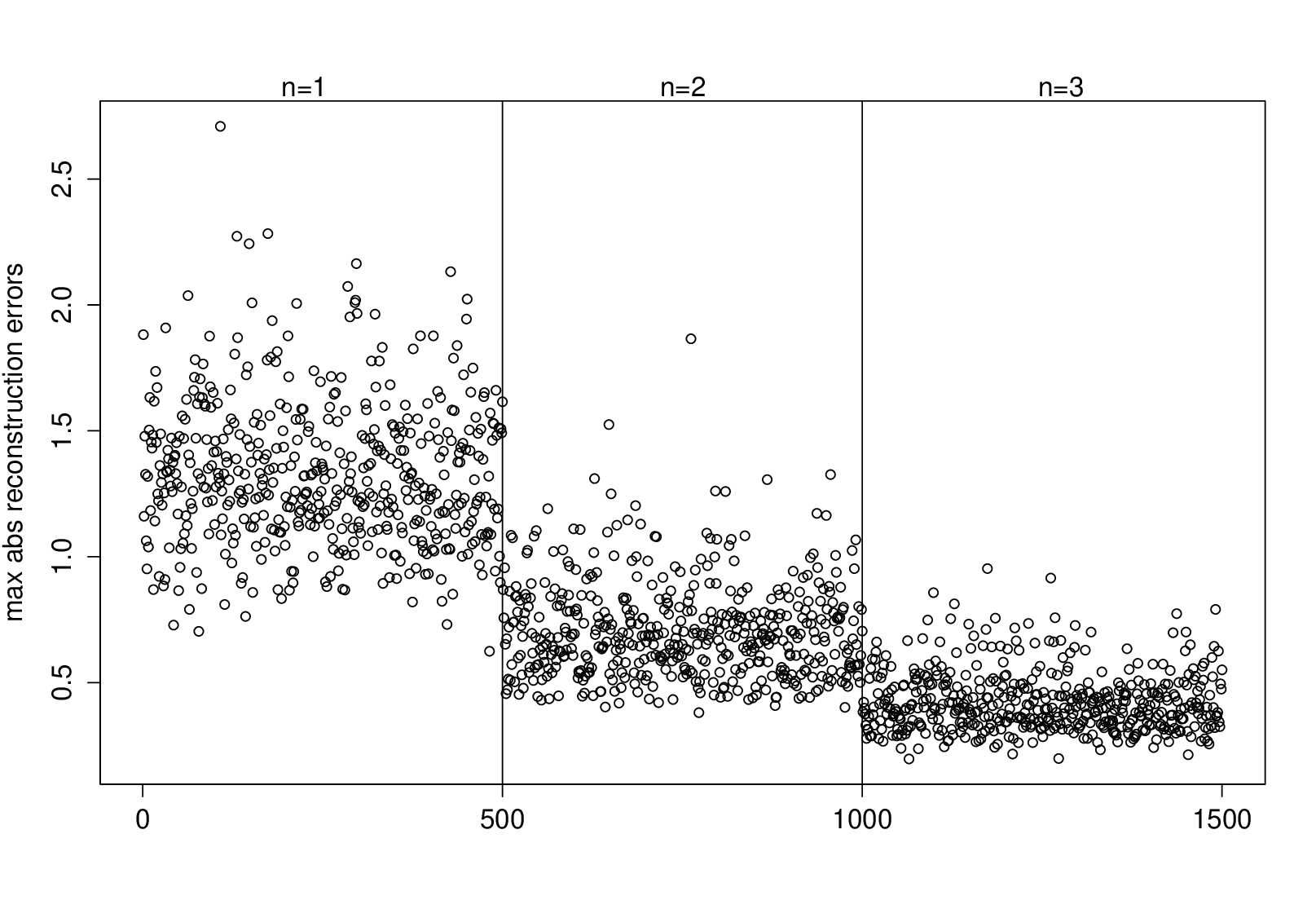}\\
 {\textbf{Fig.\,3}\ \ \small Reconstruction errors for three wavelet approximations}  \end{minipage}\quad\quad
\begin{minipage}{7.5cm}
 \includegraphics[width= 7.5cm,height=6.1cm,trim=0cm 0.5cm 0.5cm 1.6cm, clip=true]{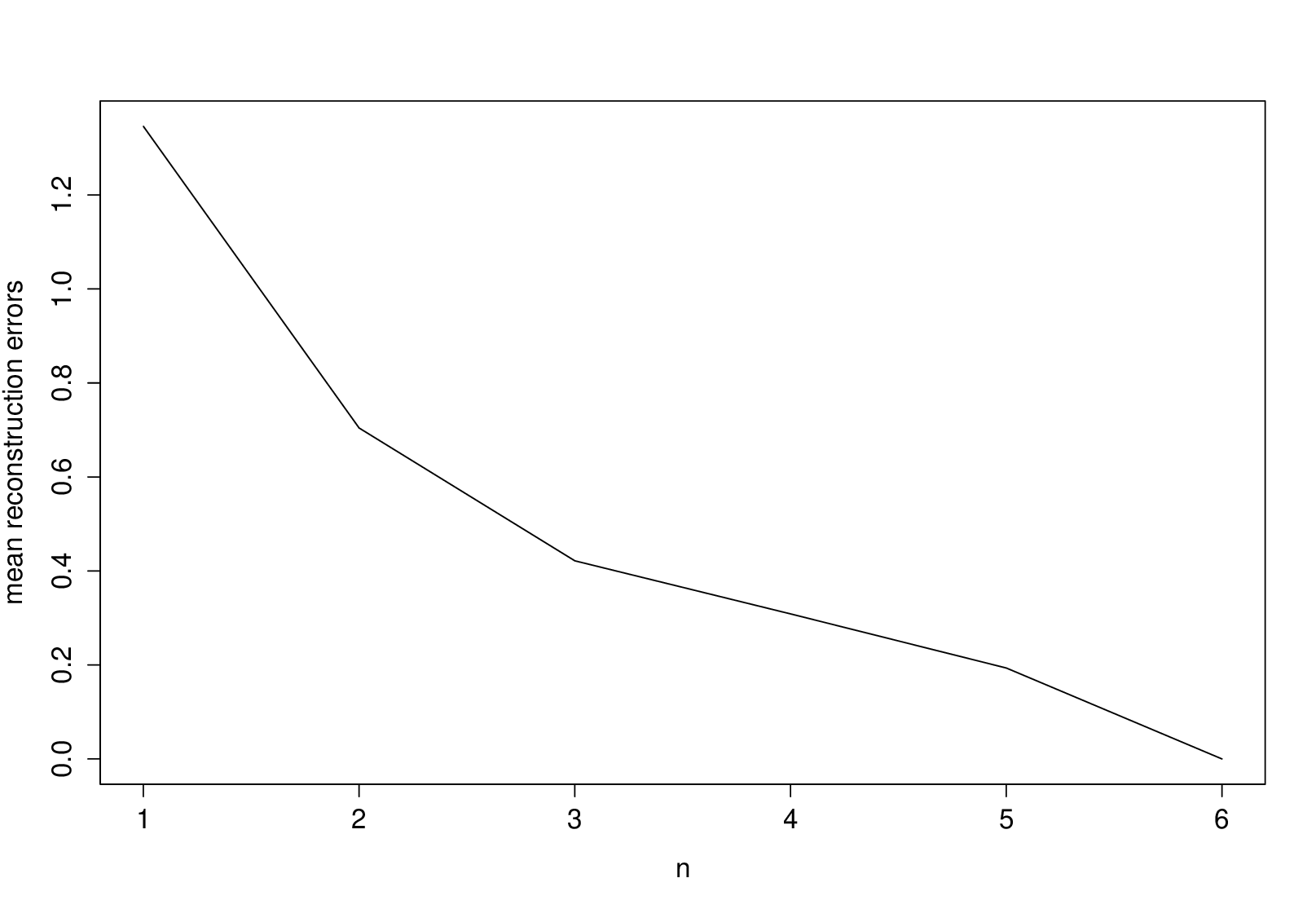}\\
{\textbf{Fig.\,4}\ \ \small  Mean reconstruction error as\\ a function of $n$} \end{minipage}
\end{figure}

\end{document}